\documentclass[11pt, twoside, leqno,draft]{article}
\usepackage{amsmath,amsthm}
\usepackage{amssymb,latexsym}
\usepackage{enumerate}
\usepackage[all]{xy}
\usepackage{amscd,multicol}
\usepackage{hyperref}
\usepackage{graphicx,color}

\newtheorem{theorem}{Theorem}
\newtheorem{prop}[theorem]{Proposition}
\newtheorem{lemma}[theorem]{Lemma}

\newtheorem{claim}{Claim}

\newtheorem*{cor}{Corollary}

\theoremstyle{definition}

\newtheorem{remark}{Remark}

\numberwithin{equation}{section}

\newcommand{\Q}{{\mathbb Q}}
\newcommand{\R}{{\mathbb R}}

\def\Cob{\text{\rm Cob}^{SO}}

\allowdisplaybreaks

\begin{document}

\title{Homotopy investigation of classifying spaces of cobordisms of singular maps}

\author{Andr\'as Sz\H{u}cs\thanks{Department of Analysis, E\"otv\"os Lor\'and University (ELTE), P\'azm\'any P. s\'et\'any I/C,  Budapest, H-1117 Hungary, email: szucs@math.elte.hu} \and Tam\'as Terpai\thanks{Department of Analysis, E\"otv\"os Lor\'and University (ELTE), P\'azm\'any P. s\'et\'any I/C,  Budapest, H-1117 Hungary, email: terpai@math.elte.hu}}

\date{}

\maketitle

\renewcommand{\thefootnote}{}

\footnotetext{Acknowledgement: Both authors were supported by the National Research, Development and Innovation Office NKFIH (OTKA) Grant K 112735; the second author was supported by the National Research, Development and Innovation Office NKFIH (OTKA) Grant K 120697.}

\footnote{2010 \emph{Mathematics Subject Classification}: Primary 57R45; Secondary 57R90.}
\footnote{\emph{Keywords and phrases}: Singularity theory, cobordism, classifying space.}

\begin{abstract}
The classifying spaces of cobordisms of singular maps have two fairly different constructions. We expose a homotopy theoretical  connection between them. As a corollary we show that the classifying spaces in some cases have a simple product structure.
\end{abstract}

\section{Introduction}

The basis of all sorts of cobordisms is the Pontrjagin-Thom construction, that reduces the computation  to homotopical investigation of the classifying space.
For cobordisms of singular maps  the construction of the classifying space is rather involved.
There are two different constructions and the aim of the present paper is to establish a connection between them.

The first construction can be called the glueing procedure, the second one provides the classification space as a fibration. The latter we describe in section 2, and here we sketch the glueing procedure, which is based on global understanding of singular maps.
This global understanding includes the following three steps:
\begin{enumerate}[1)]
\item Local singularity theory: this concerns the normal forms of map-germs. For a long time this was the main subject of singularity theory. Results include those of Whitney \cite{Whitney}, Mather \cite{Mather}, Arnold \cite{Arnold}, Morin \cite{Morin} and others.
\item Automorphism groups of the normal forms. This provides a description of a singular map of a smooth manifold into another one in the neighbourhoods of its strata (a stratum is a set of points with equal local forms) as a bundle of map-germs, the structure group of the bundle is the authomorphism group. The corresponding universal bundles can be called the ``semiglobal description''. For results see J\"anich \cite{Janich}, Wall \cite{Wall}, Rim\'anyi \cite{RRPhD}, Sz\H ucs \cite{SzucsLNM}, Rim\'anyi-Sz\H ucs \cite{RSz}.
\item The final, global level describes how the semiglobal descriptions of the different strata fit together, how the neighbourhouds of simpler strata are glued to that of a more complicated stratum.  Glueing the universal germ-bundles accordingly one  constructs the classifying space for singular maps with a given list of allowed singularities (or multisingularities).
\end{enumerate}

\section{Details of the three steps and the glueing con\-struc\-tion of the classifying space}

\subsection{The local forms}
We consider smooth maps of $n$-manifolds to $(n+k)$-manifolds for a fixed positive $k$ and arbitrary $n$. We define equivalence classes of the germs of such maps as follows.
\par
A germ $\eta: (\R^n,0) \to (\R^{n+k},0)$ is \emph{equivalent} to
\begin{itemize}
\item all germs that are right-left equivalent to it (germs of the form $\psi \circ \eta \circ \varphi$ with $\varphi$ a germ of a diffeomorphism of $(\R^n,0)$ and $\psi$ a germ of a diffeomorphism of $(\R^{n+k},0)$);
\item the germ $\eta \times id_{\R^1}: (\R^n \times \R^1,0) \to (\R^{n+k}\times \R^1,0)$.
\end{itemize}
An equivalence class will be called a \emph{local singularity} (even if it has rank $n$ at $0$ and thus is a germ of an embedding).
\par
On the set of local singularities there is a hierarchy (partial ordering): we say that $\zeta<\eta$, i.e. $\eta$ is more complicated than $\zeta$,  if in any neighbourhood of $0 \in \R^n$ a germ that belongs to $\eta$ has a point where its germ belongs to $\zeta$.
\par
Alongside the local singularities (at the points of the source manifold) we shall consider \emph{multisingularities} at the points of the target manifold. If $f: M^n\to P^{n+k}$ is a smooth map such that each point $p \in P$ has finitely many preimage points, then the multiset (set with elements equipped with multiplicities) of local singularities at the preimages of $p$ will be called the multisingularity of $f$ at $p$. One can say that a multisingularity is a formal linear combination of local singularities with positive integer coefficients. Note that there is a hierarchy on the set of multisingularities as well, induced by the hierarchy on the set of local singularities.

\subsection{The automorphism groups and the semiglobal description.}\label{section:semiglobal}
Given a local singularity $\eta$ we shall consider its minimal dimensional representative (called the \emph{root} of $\eta$, denoted by $\eta_0$). Let $Aut_\eta$ be the automorphism group of $\eta_0: (\R^c,0)\to (\R^{c+k},0)$:
\begin{align*}
Aut_\eta =\{ (\varphi,\psi) : & \varphi \mbox{ and }\psi \mbox{ are germs of diffeomorphisms of }(\R^c,0) \\
&  \mbox{and }(\R^{c+k},0) \mbox{, respectively, such that } \eta_0 \circ \varphi = \psi \circ \eta_0 \}.
\end{align*}
J\"anich \cite{Janich} and Wall \cite{Wall} defined and proved the existence of the maximal compact subgroup of such an automorphism group $Aut_\eta$; we denote it by $G_\eta$. The group $G_\eta$ is constructed as a linear group and inherits a natural representation $\lambda$ on $(\R^c,0)$ as well as $\tilde \lambda$ on $(\R^{c+k},0)$:
$$
{\lambda (\varphi, \psi) = \varphi}, \ {\tilde \lambda(\varphi, \psi) = \psi}.
$$
One can suppose that the images of $\lambda$ and $\tilde \lambda$ are subgroups of the orthogonal groups $O(c)$ and $O(c+k)$, respectively.
\par
The semiglobal description of the $\eta$-stratum consists of the following objects:
\begin{enumerate}[a)]
\item The group $G_\eta$ and its representations $\lambda: G_\eta \to O(c)$ and $\tilde \lambda : G_\eta \to O(c+k)$.
\item The universal vector bundles $\xi_\eta$ and $\tilde\xi_\eta$ associated to the representations $\lambda$ and $\tilde\lambda$, respectively:
$$
\xi_\eta = EG_\eta \underset{\lambda}{\times} \R^c, \quad \tilde \xi_\eta = EG_\eta \underset{\tilde\lambda}{\times} \R^{c+k}
$$
(here $EG_\eta$ is a contractible free $G_\eta$-space).
\item A fiberwise map $\Phi_\eta : \xi_\eta \to \tilde\xi_\eta$ such that on each fiber the restriction belongs to the right-left equivalence class of $\eta_0$.
\end{enumerate}
The map $\Phi_\eta$ can be called  {\it{the semiglobal description of $\eta$}} because it describes the map around the stratum of $\eta$-points. Namely for any map $f: M^n \to P^{n+k}$ that has no more complicated multisingularities than $1\cdot \eta$ there are tubular neighbourhoods $T$ of the (closed, smooth) $\eta$ stratum embedded in $M$ and $\tilde T$ of the image of the $\eta$ stratum embedded into $P$ such that there is a commutative diagram
$$
\xymatrix{
T \ar[r]\ar[d]_{f|_T} & \xi_\eta \ar[d]^{\Phi_\eta} \\
\tilde T \ar[r] & \tilde\xi_\eta \\
}
$$
This is why we can say that the map $\Phi_{\eta}: \xi_{\eta} \to \tilde \xi_{\eta}$ is the universal object for the maps of a tubular neighbourhood of the $\eta$-stratum into that of its image. After necessary modifications a similar semiglobal description can be given for each multisingularity stratum (in particular, the map $\Phi_\eta$ can also be defined when $\eta$ is any multisingularity).

\subsection{Inductive constructions of the classifying space}

\subsubsection{The glueing procedure}\label{section:block}

Let $\tau$ be a set of multisingularities such that whenever $\eta\in\tau$ and $\zeta<\eta$, we also have $\zeta\in\tau$, and additionally for any given bound $N$ the number of multisingularities in $\tau$ that have codimension at most $N$ (in the source, say) is finite. A smooth map is called a \emph{$\tau$-map} if all its multisingularities belong to $\tau$. We now construct a universal $\tau$-map, whose target will be called {\emph{the classifying space for $\tau$-maps}} (the universality of this map is meant in the sense of Theorem \ref{thm:universal}).
\par
First we assume that $\tau$ is finite. The construction proceeds by induction. We start with the simplest case, when $\tau$-maps are embeddings. Then the universal map is $BO(k) \hookrightarrow MO(k)$, the inclusion of the Grassmannian manifold into the Thom space of the universal rank $k$ vector bundle. Let $\eta$ be a maximal multisingularity in $\tau$ and let $\tau'$ be the multisingularity set $\tau \setminus \{\eta\}$. Assume that the universal $\tau'$-map $f_{\tau'}: Y_{\tau'} \to X_{\tau'}$ has already been constructed. Consider the map
$$
\Phi_\eta:D(\xi_\eta) \to D(\tilde\xi_\eta)
$$
of the disc bundles. Here $D(\tilde\xi_\eta)$ is the disc bundle of $\tilde\xi_\eta$ (of a sufficiently small radius with respect to a smooth metric) and $D(\xi_\eta)$ is the preimage $\Phi^{-1}_\eta(D(\tilde\xi_\eta))$, which is homeomorphic to the disc bundle of $\xi_\eta$. In particular, the boundary $S(\xi_\eta) = \partial D(\xi_\eta)$ is mapped into $S(\tilde\xi_\eta)$ and the restriction of $\Phi_\eta$ to this boundary is a $\tau'$-map (it has no $\eta$-points). The universality of $f_{\tau'}$ implies that the map $\Phi_\eta|_{S(\xi_\eta)}: S(\xi_\eta) \to S(\tilde\xi_\eta)$ can be induced from it in the sense that there is a commutative diagram
$$
\xymatrix{
S(\xi_\eta) \ar[r]^{\Phi_\eta|_{S(\xi_\eta)}} \ar[d]_{\rho} & S(\tilde\xi_\eta) \ar[d]^{\tilde \rho_\eta} \\
Y_{\tau'} \ar[r]^{f_{\tau'}} & X_{\tau'}
}
$$
We define
\begin{align*}
Y_\tau &= Y_{\tau'} \underset{\rho_\eta}{\cup} D(\xi_\eta)\\
X_\tau &= X_{\tau'} \underset{\tilde\rho_\eta}{\cup} D(\tilde\xi_\eta)\\
f_\tau &= f_{\tau'} \cup \left( \Phi_\eta|_{D(\xi_\eta)}\right)
\end{align*}
Should $\tau$ be infinite, we define $\tau[N]$ as the set of those elements $\eta$ of $\tau$ whose codimension in the source is at most $N$ and assemble $X_\tau$, $Y_\tau$ and $f_\tau$ as the direct limits of the constructions for $X_{\tau[N]}$, $Y_{\tau[N]}$ and $f_{\tau[N]}$, respectively, as $N\to \infty$.

\begin{theorem}[{\cite{RSz},\cite{GT}}]\label{thm:universal}
The map $f_\tau: Y_\tau \to X_\tau$ is the universal $\tau$-map in the sense that whenever $f:M \to P$ is a $\tau$-map, there is a homotopically unique map $\kappa_f:P\to X_\tau$ and a map $\hat \kappa :M \to Y_\tau$ such that the diagram
$$
\xymatrix{
M \ar[r]^f \ar[d]^{\hat\kappa} & P \ar[d]^{\kappa_f} \\
Y_\tau \ar[r]^{f_\tau} & X_\tau 
}
$$
commutes. \qed
\end{theorem}

Given a closed $(n+k)$-manifold $P$, we denote by $Cob_\tau(P)$ the set of cobordism classes of $\tau$-maps of closed $n$-manifolds into $P$ (with cobordisms being $\tau$-maps into $P\times [0,1]$).

\begin{cor}[{\cite{RSz},\cite{GT}}]
$$
{\rm Cob}_\tau(P) = [P,X_\tau].
$$
\end{cor}

This explains why $X_\tau$ is called the classifying space for $\tau$-maps. For futher purpose we make the following remark.

\begin{remark}\label{remark:gener}
The universal $\tau$-map is not literally a $\tau$-map, since its domain and target are not finite dimensional manifolds, but it is a limit of (proper) $\tau$-maps. We shall call such a map {\emph{generalized $\tau$-map}}. For example, the product
 $f_\tau \times (\text{identity of} \ S) : Y_\tau \times S \to X_\tau \times S$, where $S$ is a limit of finite dimensional manifolds, is also  a generalized $\tau$-map.
Theorem \ref{thm:universal} implies in particular that any $\tau$-map induces a map of its target manifold into the classifying space $X_\tau$. This remains true for generalized $\tau$-maps as well. For example, the universal $\tau$- map $f_\tau:Y_\tau \to X_\tau$ corresponds to the identity map $X_\tau \to X_\tau$, while a product map
$f_\tau \times (\text{identity of} \ S)$ corresponds to the projection map $X_\tau \times S \to X_\tau$. Note also that the union of $\tau$-maps (or generalized $\tau$-maps) corresponds to the pointwise product of the corresponding classifying maps in the $H$-space $X_\tau$ (see \cite{GT}).
\end{remark}

\section{Fibration construction of the classifying space}\label{section:fibration}

We now discuss the second construction of the classifying space, introduced in \cite{GT}. This construction works only under the assumption that $\tau$ consist of all multisingularities that are linear combinations of  local singularities from a fixed finite set of local singularities.

From now on we use the following notation:
\begin{itemize}
\item $\eta$ is a maximal local singularity from this fixed set $\tau$
\item $\tau'$ is the set of all multisingularities in $\tau$ that do not contain $\eta$
\end{itemize}
The construction proceeds as follows: assume that the classifying space $X_{\tau'}$ has already been constructed. Let $T\tilde\xi_\eta$ denote the Thom space of the bundle $\tilde\xi_\eta$ and let $\Gamma T\tilde\xi_\eta$ denote the space $\Omega^\infty S^\infty T\tilde\xi_\eta = \lim_{m \to \infty} \Omega^m S^m T\tilde\xi_\eta$ (where $\Omega A$ is the loop space of the space $A$ and $SA$ is the suspension of $A$). Note that $\Gamma T\tilde\xi_\eta$ is the classifying space for cobordims of  immersions equipped with a pullback of the normal bundle from $\tilde\xi_\eta$.
\par
The induction step for this construction is the observation (conjectured by E. Szab\'o and proved in \cite{GT}, see also another proof in \cite{altFibration}) that there is a fibration (we shall call it the ``key bundle'')
\begin{equation}\label{eq:keyBundle}\tag{$\heartsuit$}
\begin{gathered}
\xymatrix{
X_{\tau'} \ar@{^{(}->}[r] & X_\tau \ar[d]^{p_\tau} \\
& \Gamma T\tilde\xi_\eta
}
\end{gathered}
\end{equation}
i.e. $X_\tau$ is the total space, $X_{\tau'}$ is the fibre, and $\Gamma T\tilde\xi_\eta$ is the base space.
\par
The map $p_\tau$ can be described by the natural map between the functors $P \mapsto [P,X_\tau]$ and $P\mapsto [P,\Gamma T\tilde\xi_\eta]$ as follows. By Corollary to Theorem \ref{thm:universal} the set $[P,X_\tau]$ is in one-to-one correspondence with the cobordism group ${\rm Cob}_\tau(P)$ of $\tau$-maps in $P$. There is a well-known geometric interpretation of the set of homotopy classes $[P,\Gamma T\zeta]$ for any vector bundle $\zeta$, see e.g. \cite{Wells} for the case of $\zeta$ being the universal $k$-bundle -- it is in one-to-one correspondence with the cobordism group ${\rm Imm}^\zeta(P)$ of immersions with target $P$ and normal bundle pulled back from $\zeta$ (such immersions are called $\zeta$-immersions). Considering the image of the $\eta$-stratum of a $\tau$-map (which is immersed in the target manifold), it has a normal bundle that admits a natural pullback from $\tilde\xi_\eta$. Hence we obtain a natural map ${\rm Cob}_\tau (P) \to {\rm Imm}^{\tilde\xi_\eta}(P)$ that associates to the cobordism class of a $\tau$-map the cobordism class of the $\tilde\xi_\eta$-immersion of its $\eta$-stratum. This natural map is induced by a map of the corresponding classifying spaces, and that is $p\tau:X_\tau\to\Gamma T\tilde\xi_\eta$.
\par
A second description of $p_\tau$ (that is crucial in the proof of Theorem \ref{thm:universal}) comes from the glueing procedure that produces $X_\tau$. For each multisingularity $\theta=J+ n\eta$, where $J$ is a multisingularity that belongs to $\tau'$, we consider the disc bundle $D_\theta=D_J \times \left(D(\tilde\xi_\eta)\right)^n$ -- here $D_J$ for $J=\sum_i m_i \eta_i$ denotes the product $D_J= \prod_i \left(D(\tilde\xi_{\eta_i})\right)^{m_i}$. Then we attach them using maps $\tilde \rho_\theta$ (section \ref{section:block} describes them for $\theta=1\cdot\eta$) along their boundary. If we project each such disc bundle onto the factor $(D_\eta)^n$ and attach these bundles correspondingly, then we obtain $\Gamma T\tilde\xi_\eta$. These projections are compatible with the attaching maps and hence combine to a map $X_\tau \to \Gamma T\tilde\xi_\eta$, which is $p_\tau$.
\par
The key bundle can be induced from a universal $X_{\tau'}$-bundle
$$
\xymatrix{
X_{\tau'} \ar@{^{(}->}[r] & EX_{\tau'} \ar[d] \\
&  B X_{\tau'}}
$$
where $EX_{\tau'}$ is a contractible space
 and $BX_{\tau'}$ is a space such that its loop space $\Omega BX_{\tau'}$ is $X_{\tau'}$.
The main result of the paper expresses the map $b_\eta:  \Gamma T\tilde\xi_\eta \to BX_{\tau'}$ that induces the key bundle from the universal $X_{\tau'}$-bundle through the glueing map $\tilde \rho_\eta: \partial (D\tilde \xi_\eta) \to X_{\tau'}$ of section \ref{section:block}, and in this way we connect the two constructions. Namely, we give a purely homotopy theoretical construction that taking as an input the attaching map $\tilde \rho_\eta$ (which has a clear geometric meaning in the first construction) produces the map $b_\eta$ that induces the key bundle  $X_\tau \to \Gamma  T\tilde \xi_\eta$ and thus allows translating questions related to the latter construction to questions relating to the former one (see the corollaries in section \ref{section:apps}).

\section{Formulation of the result}

In this section we outline the general structure of the proof of our main result and encapsulate its key points in several claims. The proofs of the claims are deferred to the next section.

First we describe the key bundle $p_\tau:X_\tau \to \Gamma T\tilde\xi_\eta$ over the open ball bundle $\mathring D \tilde\xi_\eta \subset T\tilde\xi_\eta \subset \Gamma T \tilde\xi_\eta$, which is the Thom space $T\tilde\xi_\eta$ of $\tilde\xi_\eta$ without its special point.

\begin{claim}\label{claim:b_D}
$$
p_\tau^{-1}(\mathring D \tilde\xi_\eta) = X_{\tau'} \times \mathring D \tilde\xi_\eta.
$$
\end{claim}

Next we describe the key bundle over the Thom space $T \tilde\xi_\eta$ and its inducing map. Denote by $X_T$ the space $p^{-1}_\tau(T\tilde\xi_\eta)$ and consider the bundle $p_\tau|_{X_T}: X_T \overset{X_{\tau'}}{\to} T\tilde \xi_\eta$. Its inducing map, $b_T: T\tilde\xi_\eta \to BX_{\tau'}$, is the restriction of $b_\eta$ to $T \tilde \xi_\eta$.
Let $pr:T\tilde\xi_\eta \to T\tilde \xi_\eta / BG_\eta$ be the quotient map that contracts the zero section of $\tilde \xi_\eta$. Note that the quotient space $T\tilde\xi_\eta/BG_\eta$ coincides with the suspension $\Sigma S(\tilde\xi_\eta)$ of the sphere bundle $S(\tilde\xi_\eta)$. Since $p_\tau$ is trivial over the open ball bundle $\mathring D \tilde\xi_\eta$, in particular over its zero section $BG_\eta$, after a homotopy it can be supposed that the map $b_T$ maps $BG_\eta$ to a single point. Hence the map $b_T$ can be decomposed into the composition of the quotient map  $pr:T\tilde\xi_\eta \to  T\tilde \xi_\eta / BG_\eta$ and  a map $\sigma:  T\tilde \xi_\eta / BG_\eta  \to BX_{\tau'}$, i.e. $b_T = \sigma \circ pr$.
\par
To express the map $\sigma$, recall that for any pointed topological spaces $A$ and $B$ the (based) maps $A \to \Omega B$ are in one-to-one correspondence with (based) maps $\Sigma A \to B$, where $\Sigma$ is the reduced suspension. We call such maps that correspond to one another \emph{adjoint}. In our case, $T\tilde \xi_\eta / BG_\eta = \Sigma(S(\tilde\xi_\eta))$ and $X_{\tau'}=\Omega BX_{\tau'}$, allowing the formulation of the following claim:

\begin{claim}\label{claim:adjoint}
$\sigma:\Sigma S(\tilde\xi_\eta) \to BX_{\tau'}$ is the adjoint of $\tilde \rho_\eta: S(\tilde\xi_\eta) \to X_{\tau'}$.
\end{claim}

We have hence expressed the restriction $b_T = b_\eta|_{T\tilde \xi_\eta}$ using the attaching map $\tilde \rho_\eta$.

We can now describe the space $X_\tau$. In order to do so, recall that Brown's representability theorem admits a variant (e.g. \cite[Ch 1, \S 5, Theorem 6]{Spanier}) that states that if a functor $F$ represented by a space $B^F$ takes values in groups, then $B^F$ is an H-space. Since ${\rm Cob}_\tau(P)$ is a group (for those $\tau$ that are considered in this construction of classifying spaces -- see beginning of Section \ref{section:fibration}), the space $X_\tau$ is an H-space. The operation in $X_\tau$ corresponds to that in ${\rm Cob}_\tau(P)$, which is induced by disjoint union of the sources of the $\tau$-map representatives.

\begin{claim}\label{claim:plus1}
The space $p^{-1}_\tau(T\tilde\xi_\eta) = X_T$ is obtained from the disjoint union $D(\tilde\xi_\eta) \times X_{\tau'} \sqcup X_{\tau'}$ by attaching $S(\tilde\xi_\eta) \times X_{\tau'} \subset D(\tilde\xi_\eta) \times X_{\tau'}$ to $X_{\tau'}$ by the map
\begin{equation}\label{eq:product}\tag{$\dag$}
(s,x) \mapsto \tilde\rho_\eta(s) \cdot x
\end{equation}
where $s \in S(\tilde\xi_\eta)$, $x\in X_{\tau'}$ and the dot denotes the multiplication in the $H$-space $X_{\tau'}$.
\end{claim}

Finally, in order to express the map $b_\eta$ itself we need the following claim:

\begin{claim}\label{claim:loop}
The map $b_\eta$ is an infinite loop space map.
\end{claim}

In particular we claim that both the domain and the target spaces of $b_\eta$ are infinite loop spaces. In such a case the map $b_\eta$ is completely determined by its restriction $b_T = b_\eta|_{T \tilde\xi_\eta}$ due to the following general statement:

\begin{prop}[{\cite[p. 39.]{CLM}, \cite[pp.42--43.]{May}}]\label{prop:extension}
Let $K$ be an infinite loop space, $Y$ any $CW$ complex, and $g: Y \to K$ any continuous map. Then there is a homotopically unique infinite loop space map $\Gamma Y \to K$ extending $g$. (Here $\Gamma Y=\Omega^\infty S^\infty Y$).
\end{prop}

This extension will be denoted by $g_{ext}$. Note that if $g$ is null-homotopic, then $g_{ext}$ is also null-homotopic.

\begin{claim}\label{claim:Gamma}
If $K$ is an infinite loop space, $f: Y \to Y'$ is an arbitrary continuous map, furthermore $g: Y \to K$ and $g': Y' \to K$ are maps such that $g=g'\circ f$, then $g_{ext} = g'_{ext} \circ \Gamma f$.
\end{claim}

Finally, we summarize the Claims in the following theorem. Let us denote by $\sigma: \Sigma S(\tilde \xi_\eta) \to B X_{\tau'}$ the adjoint of the glueing map $\tilde \rho_\eta: S(\tilde \xi_\eta) \to X_{\tau'}$ (see the first construction of $X_\tau$). Let $pr: T\tilde\xi_\eta \to T\tilde\xi_\eta / BG_\eta = \Sigma S(\tilde\xi_\eta)$ be the quotient map contracting the zero section.

\begin{theorem}\label{thm:main}
The classifying map $b_\eta: \Gamma T \tilde \xi_\eta \to B X_\tau$ has the form
$$b_\eta=(\sigma \circ pr)_{ext}.$$
\end{theorem}

\section{Proofs of the Claims}

\subsection{Proof of Claim \ref{claim:b_D}}

The proof follows from the attaching procedure that gives $X_{\tau'}$ and $X_{\tau}$ (the second description of $p_\tau$ from section \ref{section:fibration}). To construct $X_{\tau'}$ one starts with a one-point space and attaches disc bundles $D_J$, one for each allowed multisingularity $\sum_{i=0}^{r-1} m_i\eta_i$ in $\tau'=\{\eta_0, \dots, \eta_{r-1}\}$ corresponding to the multiindex $J=(m_0,\dots,m_{r-1})$, in an order compatible with the partial order of these multisingularities. Denote the attaching map of $D_J$ by $r_J$. The space $p^{-1}_\tau\left(\mathring D(\tilde\xi_\eta)\right)$ is obtained by starting with $\mathring D(\tilde\xi_\eta)$ and attaching products $\mathring D (\tilde\xi_\eta) \times D_J$ in the same order using the maps $id_{\mathring D(\tilde\xi_\eta)} \times r_J$ (where $id_{\mathring D(\tilde\xi_\eta)}$ is the identity map of $\mathring D(\tilde\xi_\eta)$). The Claim follows. \qed


\subsection{Proof of Claim \ref{claim:adjoint}}

Next we try to understand the $X_{\tau'}$-bundle over $\Sigma S(\tilde \xi_\eta)$ induced by the map $\sigma: \Sigma S(\xi_\eta) \to BX_{\tau'}$. We shall denote this bundle by $\hat\xi$. It can be obtained from the disjoint union
$X_{\tau'} \times C(S(\tilde \xi_\eta)) \sqcup X_\tau'$, where  $C(S(\tilde \xi_\eta))$ denotes the cone over $S(\tilde \xi_{\eta})$,  and performing the identifications
$(x,s) \in X_{\tau'} \times C(S(\tilde \xi_\eta)) \sqcup X_\tau'$ with $\tilde \rho_\eta(s) \cdot x \in X_{\tau'}$, if $s \in S(\tilde \xi_\eta)$.
We want to deduce that the map $\sigma: \Sigma S(\tilde \xi_{\eta}) \to BX_{\tau'}$ is the adjoint of $\tilde \rho_\eta: S(\tilde \xi_{\eta}) \to X_{\tau'}$ (recall that $X_{\tau'} = \Omega BX_{\tau'}$). This will follow from the following general proposition.


\begin{prop}
Let $\mathcal H$ be an $H$-space with a principal classifying bundle $\pi: E\mathcal H \overset{\mathcal H}{\to} B\mathcal H$. For any space $Z$ denote by $\mathcal H(Z)$ the set of principal $\mathcal H$-bundles over $Z$. It is well-known that for a suspension $Z=\Sigma Y$ the set $\mathcal H(Z)$ has the following bijective identifications:
\begin{enumerate}[a)]
\item $\varphi: \mathcal H(\Sigma Y) \to [\Sigma Y, B\mathcal H]$
\item $\psi: \mathcal H(\Sigma Y) \to [Y, \mathcal H]$
\end{enumerate}
Then the map $\psi\circ \varphi^{-1}$ sends each homotopy class into its adjoint.
\end{prop}

Applying this Proposition to $\mathcal H = X_{\tau'}$ and $Y=S(\tilde\xi_\eta)$ and the element $\hat\xi \in \mathcal H(\Sigma Y)$ we get that $\varphi(\hat\xi) = \sigma$ and $\psi(\hat\xi) =\tilde\rho_\eta$ are adjoint and Claim \ref{claim:adjoint} follows. The rest of the subsection is devoted to the proof of the Proposition.

\par

Consider the following commutative diagram of pairs of spaces:

$$
\xymatrix{
(Cone\ Y,Y) \ar[r]^{\tilde f} \ar[d]_{pr}^{Y \times \{1\}} \ar@/^2pc/[rr]^{\hat f}& (P(B\mathcal H),\Omega(B\mathcal H)) \ar[d]^{\Omega B\mathcal H}  \ar[r]^>>>>>{\cong} & (E\mathcal H,\mathcal H) \ar[d]^{\mathcal H} \\
(\Sigma Y,\star) \ar[r]^f & (B\mathcal H,\star) \ar[r]^{id_{B\mathcal H}} & (B\mathcal H,\star)
}
$$

Here $P(B\mathcal H) \to B\mathcal H$ is the path-bundle ($P(B\mathcal H)$ is the space of paths starting at a fixed point $\star\in B \mathcal H$), $f:\Sigma Y \to B \mathcal H$ is a map that induces a bundle $\beta$ over $\Sigma Y$. We can and do assume that $f$ is an embedding. The map $pr: Cone \ Y = Y \times [0,1]/Y \times \{ 0 \} \to \Sigma Y= Cone\ Y / Y \times \{ 1 \}$ is the quotient map. $\tilde f$ is the lift of $f$ that maps $(y,t)$ to the path $f|_{\{y\} \times [0,t]}$. Then $\tilde f$ maps $Y \times \{1\}$ into $\Omega B\mathcal H$ by $f^{adj}$, the adjoint of $f$. $i:\Omega B \mathcal H \to \mathcal H$ is a homotopy equivalence. If we lift the universal bundle $\pi: E\mathcal H \to B\mathcal H$ to its contractible total space $E\mathcal H$, then we get the trivial bundle $E\mathcal H \times \mathcal H$; conversely, identifying $(e,h)\in E\mathcal H \times \mathcal H$ with $(e\cdot h^{-1},1)$ (where $\cdot$ is the $H$-space product in $\mathcal H$) recovers the original bundle $\pi$. Hence the bundle $f^*\pi$, the pullback of $\pi$ by $f$, can be obtained by pulling back the trivial bundle by $\hat f : Cone \ Y \to E\mathcal H$ and then performing the identification on $Y \times \{1 \} \times \mathcal H$:
$$
(y\times \{1\},h ) \sim (y \times \{1\}, i \circ f^{adj}(y\times\{1\})\cdot h).
$$
Hence $\varphi (f^* \pi) =f$ and $\psi(f^* \pi) = f^{adj}$.

\subsection{Proof of Claim \ref{claim:plus1}}

\begin{proof}

$X_T$ is obtained  from the disjoint union $X_{\tau'} \times D(\tilde\xi_\eta) \sqcup  X_{\tau'}$ by an attaching map
$X_{\tau'} \times S(\tilde\xi_\eta) \to X_{\tau'}$, which is a homeomorphism on  $X_{\tau'} \times \{s\}$ for each $s \in
S(\tilde\xi_\eta)$.

Indeed, this map is induced by the generalized ${\tau'}$ -map
$Y_{\tau'} \times S(\tilde\xi_\eta) \sqcup X_{\tau'} \times S(\xi_\eta) \to X_{\tau'} \times S(\tilde\xi_\eta) $, which is
$f_{\tau'} \times (\text{identity of} \ S(\tilde\xi_\eta)$) on the first member of the union, and it is (identity of $(X_{\tau'})) \times (\Phi_\eta|_{S(\xi_\eta)})$ on the second member.
What would be the induced map if we had only one member of this union? The first member would induce just the projection
map $X_{\tau'} \times S(\tilde\xi_\eta) \to X_{\tau'} $ , the second one would induce $\tilde \rho_\eta \circ pr_2$, where $pr_2$ is the projection $X_{\tau'} \times S(\tilde\xi_\eta) \to S(\tilde\xi_\eta)$. (See Remark \ref{remark:gener}.) Now the whole union induces the pointwise product of the two induced maps, i.e. the map that sends $(x,s) \in X_{\tau'} \times S(\tilde\xi_\eta)$ to $(x \cdot  \tilde \rho_\eta(s)).$ This means that $X_T$ is obtained as it was claimed.

\end{proof}

\subsection{Proof of Claim \ref{claim:loop}}
We have to show that the map $\Gamma T\tilde\xi_\eta \to B X_{\tau'}$ that induces the key bundle $X_\tau \to \Gamma T\tilde\xi_\eta$ is an infinite loop space map. Our argument will use the notion of \emph{$m$-framed $\tau$-maps}, where $m$ is a nonnegative integer, that is, $\tau$-maps equipped with $m$ independent normal vector fields; we also call them \emph{$\tau \oplus \varepsilon^m$-maps} (see \cite[Definition 7]{GT} for $m=1$ and \cite[Remark 8 b)]{GT} for arbitrary $m$).

We denote the classifying space for $\tau \oplus \varepsilon^m$-maps by $X_{\tau\oplus\varepsilon^m}$. In \cite[Remark 17 c,d]{GT} it is shown that $\Omega^mX_{\tau\oplus\varepsilon^m}\cong X_\tau$. This identification equips $X_\tau$ with an infinite loop space structure. It is natural to denote $X_{\tau\oplus\varepsilon^1}$ by $BX_\tau$ since its path-fibration has fiber $X_\tau$ and a contractible total space. Similarly we write $B^mX_\tau$ for $X_{\tau\oplus\varepsilon^m}$.

Let us consider the key bundle for $\tau \oplus \varepsilon^m$-maps and the map that induces it from the universal bundle:
\begin{equation}\label{eq:Binducing}\tag{$\star$}
\begin{gathered}
\xymatrix{
B^m X_\tau \ar[r] \ar[d]& \star \ar[d]^{B^m X_{\tau'}} \\
\Gamma T(\tilde\xi_\eta \oplus \varepsilon^m) = \Gamma S^m T\tilde\xi_\eta \ar[r] & B^{m+1} X_{\tau'}
}
\end{gathered}
\end{equation}
Let us consider the resolvents for both vertical maps and the map between these resolvents induced by the horizontal maps. We will use the fact that $\Omega \Gamma SY = \Gamma Y$ for any space $Y$.
$$
\xymatrix{
\Omega B^m X_\tau \ar[r]^{p^{(m-1)}_{\tau}} \ar[d]& \Omega\Gamma S^m T\tilde\xi_\eta \ar@{}[r]|{=} \ar[d]^{b^{(m-1)}_\eta}&\Gamma S^{m-1} T\tilde\xi_\eta \ar[d]\ar[r] & B^m X_{\tau'} \ar[d]^{id}\ar[r] & B^m X_\tau \ar[d]\ar[r]^{p^{(m)}_\tau} & \Gamma S^m T\tilde\xi_\eta \ar[d]^{b^{(m)}_\eta}\\
\star \ar[r]& \Omega B^{m+1} X_{\tau'}\ar@{}[r]|{=}&B^m X_{\tau'} \ar[r] & B^mX_{\tau'} \ar[r] & \star\ar[r] & B^{m+1}X_{\tau'}&
}
$$
(in the diagram we write the resolvents horizontally and the induced maps vertically). We see that the key bundle for $\tau\oplus \varepsilon^{m-1}$-maps,
$$
p^{(m-1)}_\tau : B^{m-1} X_\tau \to \Gamma  S^{m-1} T\tilde\xi_\eta,
$$
is the loop map of the map
$$
p^{(m)}_\tau : B^{m} X_\tau \to \Gamma  S^{m} T\tilde\xi_\eta,
$$
i.e. $p^{(m-1)}_\tau = \Omega p^{(m)}_\tau$. Similarly for the inducing maps it holds that $b^{(m-1)}_\eta = \Omega b^{(m)}_\eta$.

\subsection{Proof of Claim \ref{claim:Gamma}}
We have to show that the diagram remains commutative after adding the dotted arrow.
$$
\xymatrix{
Y \ar[rr]^f \ar[dr]^g \ar[dd]_{i} & & Y' \ar[dl]_{g'} \ar[dd]^{i'} \\
& K & \\
\Gamma Y \ar[ur]^{g_{ext}} \ar@{-->}[rr]^{\Gamma f} & & \Gamma Y' \ar[ul]_{g'_{ext}}
}
$$
The square is commutative by construction of the map $\Gamma f$. Hence the composition $g'_{ext} \circ \Gamma f \circ i$ coincides with $g'_{ext} \circ i' \circ f = g' \circ f = g$, i.e. $g'_{ext} \circ \Gamma f$ is an extension of $g$ to $\Gamma Y$. The map $g'_{ext} \circ \Gamma f$ is also an infinite loop space map as a composition of infinite loop space maps, thus by uniqueness of the extension $g'_{ext} \circ \Gamma f = g_{ext}$ as claimed.

This finishes the proof of Theorem \ref{thm:main}.

\section{Applications}\label{section:apps}

Let $\tau, \eta, \tau'$ be as in section \ref{section:fibration} ($\eta$ is a maximal local singularity from $\tau$ and $\tau'$ is the set of all multisingularities in $\tau$ that do not contain $\eta$) and consider the key bundle \eqref{eq:keyBundle}. We give conditions that imply that this bundle is rationally trivial.

\begin{lemma}

If $\eta$ is a singularity for which the bundle $\tilde \xi_\eta$ is orientable and the rational cohomology ring of its Thom space $H^*(T\tilde \xi_\eta;\Q)$ has no zero divisor (equivalently: the rational Euler class of the bundle $\tilde \xi_\eta$ is non-trivial) then the key bundle is rationally trivial, i.e.
$$X_\tau \underset{\Q} {\cong} X_{\tau'} \times \Gamma T\tilde \xi_\eta.$$
\end{lemma}

\begin{proof}	
(This result was proved in \cite[Theorem 9]{GT} in a partially different way.)
It is enough to show that the map $b_\eta$ is rationally null-homotopic. Since $b_\eta =  (b_T)_{ext}$ it is enough to show that $b_T$ is rationally null-homotopic. Since the space $BX_{\tau'}$ is an $H$-space, it is rationally homotopy equivalent to product of some Eilenberg-MacLane spaces. Hence the map $b_T: T\tilde \xi_\eta \to BX_{\tau'}$ is rationally null-homotopic if it induces the trivial homomorphism in the rational cohomologies. By the Thom isomorphism any element of $H^*(T\tilde \xi_\eta;\Q)$ has the form $U\cup x$, where $U$ is the Thom class, and $x$ is an element of $H^*(BG_\eta;\Q).$ Note that
the ring $H^*(BG_\eta;\Q)$ has no zero divisor. Indeed, this is the invariant part (under the action of the Weyl goup) of
$H^*(B\mathbb T_\eta;\Q),$ where $\mathbb T_\eta$ is the maximal torus of the group $G_\eta$, and $H^*(B\mathbb T_\eta;\Q)$ is a polynomial ring.

Let us consider the cohomological exact sequence with rational coefficients of the pair $(T\tilde \xi_\eta, BG_\eta)$. Since $T\tilde \xi_\eta/BG_\eta = \Sigma S(\xi_\eta)$, a fragment of the cohomological exact sequence has the form:

$$
H^*(\Sigma S(\xi_\eta)) \to H^*(T\tilde \xi_\eta) \to H^*(BG_\eta)
$$
The latter homomorphism maps an arbitrary class $U\cup x$ into $e(\tilde\xi_\eta) \cup x$, where $e(\tilde\xi_\eta)$ is the Euler class of the bundle $\tilde \xi_\eta.$ Since $e(\tilde\xi_\eta) \cup x$ is non zero if $x$ is not zero, the map is injective. Hence the previous map $pr^*: H^*(\Sigma S(\xi_\eta)) \to H^*(T\tilde \xi_\eta)$ is the trivial homorphism, and then so is the homorphism induced by the composition $\sigma \circ pr$. Since the target of this map is an H-space, the map $\sigma \circ pr : T\tilde \xi_\eta \to BX_{\tau'}$ itself is rationally homotopically trivial. Then its extension $b_\eta$ is also rationally homotopically trivial, and then 
the key bundle $X_\tau \to \Gamma T\tilde \xi_\eta$ is trivial.


\end{proof}

Recall that ${\rm Cob}_\tau(\R^{n+1})$ is abbreviated to ${\rm Cob}_\tau(n)$, the codimension of the $\eta$ stratum in the domain is $c$ and ${\rm Imm}^{\tilde\xi_\eta} (n-c)$ denotes the cobordism group of immersions of $(n-c)$-dimen\-si\-o\-nal manifolds into $\R^{n+k}$ such that the normal bundle is pulled back from the bundle $\tilde \xi_\eta$.

\begin{cor}
Let $\tau, \eta, \tau'$ be as above. Then for any $n$
$$Cob_\tau(n)\otimes \Q \approx Cob_{\tau'}(n) \otimes \Q \oplus Imm^{\tilde\xi_\eta} (n-c) \otimes \Q$$
holds. \qed
\end{cor}


\begin{theorem}\label{thm:Xsplitting}
Assume that $\tilde \xi_\eta=\varepsilon^1\oplus \zeta$, the codimension $c = \operatorname{rank}
 \tilde\xi_\eta$ is even and $\pi_m(X_\tau) \otimes\Q=0$ whenever $m$ is odd. Then the classifying space $X_\tau$ is rationally homotopically equivalent to the direct product $X_{\tau'} \times \Gamma T\tilde\xi_\eta$.
\end{theorem}

\begin{proof}
We have to show that the map $b_\eta:\Gamma T\tilde\xi_\eta \to BX_{\tau'}$ that induces the key bundle $p_\tau$ is rationally null-homotopic. Consider its loop map $\Omega b_\eta : \Omega \Gamma T\tilde\xi_\eta \to \Omega BX_{\tau'}$. Since $T\tilde\xi_\eta= ST\zeta$ we can reinterpret this map as $\Omega b_\eta: \Gamma T\zeta \to \Omega B X_{\tau'}$. This is an infinite loop space map and therefore is the extension in the sense of Proposition \ref{prop:extension} of its restriction to $T\zeta$.

The space $T\zeta$ has nontrivial rational cohomology only in odd degrees, since the base space $BG_\eta$ has rational cohomology only in even degrees and the Thom class of the bundle $\zeta$ has odd degree $c-1$. The space $X_{\tau'}$ is an $H$-space and therefore is rationally homotopically equivalent to a product of Eilenberg-MacLane spaces $K(\Q,j)$ with some integers $j$. None of these numbers $j$ can be odd since $\pi_m(X_\tau) \otimes \Q=0$ if $m$ is odd. Hence any map $T\zeta \to X_{\tau'}$ induces the zero homomorphism in rational cohomology and therefore is rationally null-homotopic. Then its extension $\Omega b_\eta$ is also rationally null-homotopic.

Consider now the following fragment of the resolution of $p_\tau$:
$$
\Omega X_{\tau'} \to \Omega X_\tau \to \Gamma T\zeta \xrightarrow{\Omega b_\eta} X_{\tau'}
$$
Recall that in a resolution each map induces the bundle whose projection is the previous map from the path fibration of its target. In our case the map $\Omega b_\eta$ induces the bundle $\Omega X_\tau \xrightarrow{\Omega X_\tau'} \Gamma T\zeta$ from the path fibration $* \xrightarrow{\Omega X_{\tau'}} X_{\tau'}$. Since $\Omega b_\eta$ is rationally null-homotopic, it induces a rationally trivial bundle, that is,
\begin{equation}\label{eq:splitting}
\Omega X_\tau \cong_\Q \Omega X_{\tau'} \times \Gamma T\zeta
\end{equation}
The spaces $X_\tau$, $X_{\tau'}$ and $\Gamma T\tilde\xi_\eta$ are $H$-spaces and therefore the splitting \eqref{eq:splitting} implies that the rational Eilenberg-MacLane space factors of $X_\tau$ are the multiset union of those of $X_{\tau'}$ and $\Gamma T\tilde\xi_\eta$. Hence we have $X_\tau \cong_\Q X_{\tau'} \times \Gamma T\tilde\xi_\eta$.
\end{proof}

\begin{remark}
The conditions of Theorem \ref{thm:Xsplitting} are satisfied if $\tau'$ consists of all Morin (that is, corank $1$) singularities, $\eta=III_{2,2}$ is the simplest $\Sigma^{2,0}$ singularity, the codimension $k$ of the $\tau$-maps is even and cooriented maps are considered. In this case $c=2(k+2)$ is even and the assumption $\pi_m(X_{\tau'}) \otimes \Q =0$ for $m$ odd follows from \cite[Theorem 6]{GT} and the fact that $\pi_{n+k}(X_{\tau'}) \cong \Cob_{Morin}(n,k)$.
\end{remark}

\end{document}